\documentclass{amsart}

\usepackage{amsthm}
\usepackage{amsmath, amsthm}
\usepackage{amssymb}
\usepackage{tikz-cd}
\usepackage{comment}
\DeclareMathOperator{\dom}{dom}
\DeclareMathOperator{\range}{range}
\DeclareMathOperator{\image}{''}
\DeclareMathOperator{\len}{len}
\DeclareMathOperator{\acc}{acc}
\DeclareMathOperator{\Ult}{Ult}
\DeclareMathOperator{\cf}{cf}
\DeclareMathOperator{\crit}{crit}
\DeclareMathOperator{\rank}{rank}
\DeclareMathOperator{\power}{\mathcal{P}}
\DeclareMathOperator{\Col}{Col}
\newcommand{\chang}{\twoheadrightarrow}
\newcommand{\ZF}{\mathrm{ZF}}
\newcommand{\ZFC}{\mathrm{ZFC}}
\newcommand{\Ord}{\mathrm{Ord}}
\newcommand{\SCH}{\mathrm{SCH}}
\newcommand{\otp}{\mathrm{otp}}

\newtheorem{theorem}{Theorem}
\newtheorem{lemma}[theorem]{Lemma}
\newtheorem{claim}[theorem]{Claim}
\newtheorem{remark}[theorem]{Remark}
\newtheorem{definition}[theorem]{Definition}
\newtheorem{proposition}[theorem]{Proposition}
\newtheorem{question}[theorem]{Question}
\begin{document}
\title{Prikry type forcings and the Bukovsk\'y-Dehornoy phenomena}
\author{Yair Hayut}
\address{
Einstein Institute of Mathematics \\
Edmond J. Safra Campus \\ 
The Hebrew University of Jerusalem \\
Givat Ram. Jerusalem, 9190401, Israel}
\email{yair,hayut@mail.huji.ac.il}
\thanks{This document is based on a tutorial lectures that were given in Torino, in the 8th European Set Theory Conference. 
This research was supported by the Israel Science Foundation 1967/21. 
}
\begin{abstract}
This paper is meant to present in a coherent way several instances of quite common phenomena that was first identified (independently) by Bukovsk\'y and Dehornoy. We present the basic result for Prikry type forcing and show how to extend it to the Gitik-Shraon forcing, the Extender Based Prikry forcing, Prikry forcings with interleaved collapses and Radin forcing for $o(\kappa) < \kappa^+$. 
\end{abstract}
\maketitle
In 1968, Karl Prikry introduces the Prikry forcing, a forcing notion that changes the cofinality of a regular cardinal  without collapsing cardinals, \cite{Prikry1968}. Almost ten years later Bukovsk\'{y} and Dehornoy published independently a theorem showing that a generic for the Prikry forcing can be obtained in an iterated ultrapower (see Theorem \ref{thm:BD-for-vanilla}), \cite{Bukovsky1977, Dehornoy1978}, and moreover --- the generic extension over the iterated ultrapower is exactly the intersection of the intermediate steps in the iteration.   

This theorem did not get a lot of attention and remained mostly unknown. Dehornoy found a generalizations of this theorem to Magidor forcing, \cite{Dehornoy1983} and later Hamkins generalized it to tree forcing \cite{Hamkins1997}. Woodin and Cummings develop a generalization of the theorem for the Radin forcing, which remained unpublished. 

In \cite{Sakai}, Sakai derived a version of this theorem for Prikry-like forcings based on ideals but more importantly gave a way to construct a Prikry type forcing from an iteration of elementary emebedding. This method was used recently in \cite{Aguilera} in order to construct a sophisticated forcing notion for controlling the intersection of set of indiscernibles for a set of reals $A$.

In \cite{FuchsHamkins}, Fuchs and Hamkins analysed the possible Bukovsk\'y-Dehornoy phenomena for various forcing notions (both the existence of generic over iterated ultrapower and the intersection model theorem).  
 
In this paper we would like to continue the line of investigation of Sakai, but going even further: we would like to show that almost all properties of the generic extension by a Prikry type forcings can be derived naturally by considering a variant of the Bukovsk\'{y}-Dehornoy theorem and analysing the properties of the intersection model. 

In Section \ref{section:BD-theorem}, I collected some facts which are partially folklore but not known to the wide audience regarding the Bukovsk\'{y}-Dehornoy Theorem. In particular, I illuminate the potential shift in view, focusing on the generic extension as the intersection of models instead, and deriving the properties of the Prikry forcing from it. The results in this section are not due to me, and I put the right citations when possible. 

In Section \ref{section:GS-forcing}, I derive the results of Gitik-Sharon \cite{GitikSharon2008}, using the Bukovsk\'{y}-Dehornoy approach. In Section \ref{section:EBPF} I derive a Bukovsk\'{y}-Dehornoy Theorem for the Extender Based Prikry forcing and for Prikry forcing with interleaved collapses.

In Section \ref{section:Magidor-Radin} we derive a variant of the Bukovsk\'{y}-Dehornoy Theorem for the Magidor and Radin forcing. This theorem seems to be folklore as well, but I show there how to derive the main properties of Radin forcing using this approach. The main result of the last section is the \emph{failure} of the Bukovsk\'{y}-Dehornoy Theorem for Radin forcing with $o(\kappa) = \kappa^{+}$.

\section{Prikry forcing and the Bukovsk\'y-Dehornoy Theorem}\label{section:BD-theorem}
In this section we will present the Prikry forcing from the perspective of the Bokuvsk\'{y}-Dehornoy theorem. Instead of designing a forcing notion and analysing its properties, our goal is to construct a pair of models of $\ZFC$, with the same cardinals $N_0 \subseteq N_1$ such that there is an $N_0$-regular cardinal which is singular in $N_1$.
\subsection{Iterated ultrapowers}
Recall the definition of a measure on $\kappa$.
\begin{definition}
Let $\kappa$ be a regular uncountable cardinal. 
\begin{itemize}
\item An ultrafilter $U\subseteq \mathcal{P}(\kappa)$ is $\kappa$-complete if every intersection of $<\kappa$ sets from $U$ is in $U$.
\item We say that $U \subseteq \mathcal{P}(\kappa)$ is a \emph{normal measure} if $U$ is an ultrafilter that contains all co-bounded sets and for every $f \colon A \to \kappa$ such that $A \in U$ and $f(\alpha) < \alpha$ for all $\alpha \in A$, there is $\gamma \in \range f$, such that $f^{-1}(\{\gamma\}) \in U$. 
\end{itemize}
\end{definition}
Note that every normal measure is $\kappa$-complete. 

\begin{lemma}[Scott, \cite{Scott1961}]
An ultrafilter $U$ on $\kappa$ is normal if and only if there is an elementary embedding $j \colon V \to M$, such that $M$ is well founded and $\crit j = \kappa$, and $U = \{X \subseteq \kappa \mid \kappa \in j(X)\}$. 
\end{lemma}
Throughout this paper, I will try to replace (as much as possible) the combinatorial definitions of the objects involved, such as measures and extenders, with properties of elementary embeddings. This fits better with the Bukovak\'y-Dehornoy theorem as well as will some recent works of Merimovich, for example \cite{Merimovich2011}. Nevertheless, we need to know that all our elementary embeddings have combinatorial definition. This is important in order to be able to iterated them as well as to control their continuity points. 

\begin{definition}[Kunen, \cite{Kunen1970}]
Let $U$ be a $\kappa$-complete measure in a model $V$. Let us define by induction an iteration of the ultrapower embedding. 
\[\begin{matrix}
V &=& M_0 & j_{0,0} = id \\  
M_{n + 1} &=& \Ult(M_n, j_n(U)), & j_{n,n+1} \colon M_n \to M_{n+1}   \\ 
&\vdots& \\ 
M_\omega & =& \lim \langle M_n, j_{n,m} \mid n \leq m < \omega\rangle  & j_{n,\omega} \colon M_n \to M_\omega 
\end{matrix}\]
Where $\lim$ indicates the direct limit of the directed system of embeddings.
\end{definition} 
\begin{theorem}[Gaifman, \cite{Gaifman}]
$M_\omega$ is well founded.
\end{theorem}
Let us assume that $U$ is normal, for simplicity. Let $\kappa_n = j_n(\kappa)$ for $n \leq \omega$. Let $P = \langle \kappa_n \mid n < \omega\rangle$.

Let us look at $M_\omega[P]$ --- the least $\ZFC$-model containing $M_\omega$ and $P$. This model can be defined by $\bigcup_{\alpha \in \beta \in \Ord} L_{\beta}(M_\alpha \cup \{P\})$, but as the following theorems will show us, it can also be defined without referring to the $L(A)$ construction. We will show that $M_\omega$ and $M_\omega[P]$ have the same cardinals and $\kappa_\omega$ is regular in $M_\omega$ and singular in $M_\omega[P]$.

Let us start with a few simple facts about iterated ultrapowers by a measure on $\kappa$.
\begin{claim}
\begin{enumerate}
\item For every element $x \in M_n$, there is $f \in V$, $f \colon \kappa^n \to V$ such that $x = j_n(f)(\kappa_0,\dots, \kappa_{n-1})$. 
\item For every $x \in M_\omega$ there is $n < \omega$ and $\bar{x} \in M_n$ such that $x = j_{n,\omega}(\bar x)$. In particular, there is $f \colon \kappa^n \to V$ such that $x = j_\omega(f)(\kappa_0,\dots, \kappa_{n-1})$.
\item $\crit j_{n,m} = \kappa_n$ for all $n < m \leq \omega$.
\item $\kappa_\omega = \sup \kappa_n$. 
\end{enumerate}
\end{claim}
\begin{definition}
An ordinal $\alpha$ is a \emph{continuity point} of an elementary embedding $j$, if $j(\alpha) = \sup j \image \alpha$. An ordinal $\alpha$ is a \emph{fixed point} of $j$, if $j(\alpha) = \alpha$.
\end{definition}
Since $j(\beta) \geq \beta$ for all $\beta$, every fixed point is a continuity point.
\begin{claim}\label{claim:continuity-points}
Let $n < \omega$. Every regular cardinal in $M_n$ which is not $\kappa_n$ is a continuity point of $j_{n,m}$ for all $n < m \leq \omega$.
\end{claim}
\begin{proof}
Let $j = j_{n,m}$.
If $\lambda < \kappa_n = \crit j$, then it is a fixed point. Otherwise, let $\alpha < j(\lambda)$. Then, there is a function $f\in M_n$ from some finite power of $\kappa_n$ to $\lambda$ representing $\alpha$. But then, taking $\gamma = \sup \range f < \lambda$, $\alpha < j(\gamma)$ by {\L}o\v{s} theorem.
\end{proof}

The embeddings $j_{n,\omega}$ as well as the models $M_m$ for $m \geq n$ are definable in $M_n$. In particular, $M_\omega \subseteq M_n$ as a definable subclass and $P\in M_n$. Therefore, $M_\omega[P] \subseteq M_n$. We conclude that: 
\begin{lemma}\label{lemma:BD-easy-direction}
$M_{\omega}[P] \subseteq \bigcap M_n$.
\end{lemma}
From this lemma, let us conclude that $M_\omega$ and $M_\omega[P]$ have the same cardinals.
\begin{lemma}
$M_\omega[P] \cap V_{\kappa_\omega} \subseteq M_\omega$.
\end{lemma}
\begin{proof}
Pick $X \in M_\omega[P] \cap V_{\kappa_\omega}$ and let $n$ be large enough so that $\kappa_n > \rank X$. Then, $X \in M_n$ and therefore $X = j_{n,\omega}(X) \in M_\omega$. 
\end{proof}
\begin{lemma}\label{lemma:preserving-cardinals-Prikry}
If $\lambda \neq \kappa_\omega$ is a regular cardinal in $M_\omega$, then it remains regular in $M_\omega[P]$.
\end{lemma}
\begin{proof}
Let $n < \omega$ be large enough and let $\bar{\lambda}$ be the pre-image of $\lambda$ under $j_{n,\omega}$ and let $\bar\rho$ be the pre-image of $\rho = \cf^{M_\omega[P]}(\lambda)$. So $\bar\lambda$ is regular in $M_n$ and $\bar\rho \leq \bar \lambda$. Note that $\rho \neq \kappa_\omega$, which is singular in $M_\omega[P]$ and therefore $\bar\rho\neq \kappa_n$.

Since $\bar\lambda, \bar \rho$ are regular cardinals in $M_n$, and different than $\kappa_n$, by Lemma \ref{claim:continuity-points} they are continuity points of the embedding $j_{n,\omega}$:
\[\lambda = j_{n,\omega}(\bar\lambda) = \sup \{j_{n,\omega}(\alpha) \mid \alpha < \bar \lambda\},\] 
\[\rho= j_{n,\omega}(\bar\rho) = \sup \{j_{n,\omega}(\alpha) \mid \alpha < \bar \rho\}.\] 

Thus, 
\[\bar{\rho} =\cf^{M_n} (\rho) = \cf^{M_n} (\cf^{M_\omega[P]} \lambda) = \cf^{M_n} \lambda = \bar\lambda\]
\end{proof}
So, we conclude that $M_\omega \subseteq M_\omega[P]$ have the same cardinals and $P$ witness the singularity of $\kappa_\omega$ in $M_{\omega}[P]$. 

We would like to show that $M_\omega[P]$ is a generic extension of $M_\omega$ (without pointing on an explicit forcing notion), and get some general information about the properties of the forcing. 
\begin{lemma}\label{lemma:chain-condition-Prikry}
$M_\omega[P]$ is a generic extension of $M_\omega$ using a $\kappa_\omega^+$-c.c.\ forcing notion.
\end{lemma} 
\begin{proof}
By Bukovsk\'y Theorem, \cite{Bukovsky1973}, this is equivalent to the following statement: 
For every $f \colon \alpha \to \lambda$, $f \in M_\omega[P]$, there is $g \colon \alpha \to \power(\lambda)$ in $M_\omega$ such that $\forall \zeta < \alpha, f(\zeta) \in g(\zeta)$ and $|g(\zeta)| \leq \kappa_\omega$. 

Using Lemma \ref{lemma:BD-easy-direction}, we know that $M_\omega[P] \subseteq \bigcap M_n$. Let $f \in \bigcap M_n$ be a function from $\alpha$ to $\lambda$ for some ordinals $\alpha, \lambda$. Let us define for each $n$ a function $f_n$ such that if $\zeta = j_{n,\omega}(\bar\zeta), f(\zeta) = j_{n,\omega}(\bar\xi)$ we let $f_n(\bar\zeta) = \bar\xi$ (so $f_n$ is a partial function).

Let $g_n \colon \kappa^n \to V$ represent $f_n$, so $j_n(g_n)(\kappa_0, \dots, \kappa_{n-1}) = f_n$. Finally, let:
\[\bar g (\zeta) = \{g_n(\eta)(\zeta) \mid n < \omega, \eta \in \kappa^n, \zeta \in \dom g(\eta)\},\]     
and let $g = j_\omega(\bar g)$. Let us verify that for every $\zeta < \alpha$, $f(\zeta) \in g(\zeta)$. 

Let $n$ be large enough to that $\zeta, f(\zeta)$ are in the range of $j_{n,\omega}$ and in particular the pre-image of $\zeta$ under $j_{n,\omega}$ is in the domain of $f_n$. In particular, 
\[f(\zeta) = j_{n,\omega}(f_n(\bar\zeta)) \in j_{n,\omega}(\{g_n(\eta)(\zeta) \mid \eta \in j_n(\kappa)^n\})\subseteq j_\omega(\bar g)(\zeta).\]

Since $|\bar{g}(\zeta)| \leq \kappa$ for all $\zeta$, the result follows.
\end{proof}

\begin{remark}\label{remark:chain-condition-from-width}
Let us consider an arbitrary iteration of measures, so we let $V = M_0$ and $M_{n+1} = \Ult(M_n, \mathcal{V}_n)$, with a direct limit $M_\omega$. Any model of $\ZFC$ between $M_\omega$ and $\bigcap M_n$ is a $j_{\omega}(\lambda)$-c.c.\ extension of $M_\omega$, where $\lambda$ bounds the size of the sets of the measures. More precisely, if the \emph{width} of the embedding $j_n$ is $<j_n(\lambda)$, then the forcing is $j_\omega(\lambda)$-c.c.\footnote{Recall that the width of an elementary embedding $j \colon M \to N$ between two models of $\ZFC$ is $\leq\mu$ iff for every $x \in N$ there is a set $a \in M$ with $|a| \leq \mu$ such that $x \in j(a)$. Equivalently, if for every $x \in N$ there is $f \in M$, $|\dom f| \leq \mu$ and $a \in N$ such that $j(f)(a) = x$.}
\end{remark}

Note that while this approach allows us to compute the chain condition of the forcing quite easily, it is still unclear whether one can deduce that the size of the corresponding forcing is $\leq 2^{\kappa}$. 
\begin{question}
Can we show abstractly, without referring to the Prikry forcing, that $M_{\omega}[P]$ is a generic extension of $M_\omega$ using a forcing notion of cardinality $\leq 2^{\kappa_\omega}$?
\end{question}

So far our arguments only used the inclusion $M_\omega[P]\subseteq \bigcap M_n$, and applied absoluteness. In general, the intersection of models of $\ZFC$ does not have to satisfy even $\ZF$. In this case, the model $\bigcap M_n$ is definable in every $M_n$, so we can abstractly get more information, using the methods of Set Theoretic Geology, \cite[Lemma 21]{FuchsHamkinsReits2015}. 

Fortunately, we do not need to analyse this model as, surprisingly, the Bukovsk\'y-Dehornoy Theorem shows that the intersection model is the minimal possible model:
\begin{theorem}[Bukovsk\'{y}-Dehornoy, \cite{Bukovsky1977, Dehornoy1978}]\label{thm:BD-for-vanilla}
$\bigcap M_n = M_\omega[P]$.
\end{theorem}
We include the proof as its main ideas are going to repeat throughout the other, more complicated, cases.
\begin{proof}
The proof consists of two steps. In the first step we show that the model $M_\omega[P]$ is closed under countable sequences of ordinals. In the second step we use the elementary embeddings in order to approximate an arbitrary set of ordinals from $\bigcap M_n$ using a countable sequence of sets in $M_\omega$. From this approximation, $M_\omega[P]$ can compute $X$.

\begin{lemma}\label{lemma:countable-closure-vanilla-Prikry}
$M_\omega[P]$ is closed under countable sequences of ordinals.
\end{lemma}
\begin{proof}
Let $\langle \alpha_n \mid n < \omega\rangle$ be a sequence of ordinals. Fix, in $V$, a countable sequence of functions $f_n \colon \kappa^{m_n} \to \Ord$, such that $j_\omega(f_n)(\kappa_0,\dots, \kappa_{m_n-1}) = \alpha_n$. Let $\vec f = \langle f_n \mid n < \omega\rangle$.

Then $j_\omega(\vec f) \in M_\omega$. Therefore, 
\[\langle \alpha_n \mid n < \omega\rangle = \langle j_\omega(f_n)(P \restriction m_n) \mid n < \omega\rangle \in M_\omega.\] 
\end{proof}

\begin{lemma}\label{lemma:intersection-in-vanilla-Prikry}
Let $X \in \bigcap M_n$ be a set of ordinals. Then $X \in M_\omega[P]$.
\end{lemma}
\begin{proof}
Let 
\[Y_n = \{\alpha \mid j_{n,\omega}(\alpha) \in X\}, \quad Z_n = j_{n,\omega}(Y_n).\]
By the $\sigma$-closure, $\langle Z_n \mid n < \omega\rangle \in M_\omega$. Now,
\[\beta \in X \iff \forall^* n \beta \in Z_n,\]
where $\forall^*$ means for all large $n$. Indeed, if $\beta \in \range j_{n,\omega}$ then $\beta \in X \iff \beta \in Z_n$.
\end{proof}
\end{proof}

From the proof one can extract a more concrete definition for the model $M_\omega[P]$: 
\begin{proposition}
$M_{\omega}[P]$ is the least transitive class $C$ that contains $M_\omega\cup\{ P\}$ and closed under the operations:
\begin{enumerate}
\item If $\vec g = \langle g_n \mid n < \omega\rangle \in C$ such that $\dom g_n = \kappa_\omega^n$ then $\langle g_n(P \restriction n) \mid n < \omega\rangle$ is in $C$. 
\item If $\langle X_n \mid n < \omega\rangle \in C$ then $X = \{z \mid \forall^* n, z\in X_n\}$. 
\item If $X = \langle A, E\rangle \in C$ and $E$ is well founded and extensional, then the Mostowski collapse of $X$ is in $C$.
\end{enumerate}
\end{proposition}

\section{Gitik-Sharon forcing}\label{section:GS-forcing}
Let us consider now the diagonal supercompact forcing, introduced by Gitik and Sharon in \cite{GitikSharon2008}. This forcing is using a stronger type of large cardinal axioms --- supercompactness. 
\begin{definition}
Let $\kappa \leq \lambda$ be cardinals. Let $P_\kappa \lambda$ be the set of all subsets of $\lambda$ of cardinality $<\kappa$. 

A $\kappa$-complete ultrafilter $\mathcal{U}$ on $P_\kappa\lambda$ is \emph{fine} if for every $\alpha \in \lambda$, the set $\{x \in P_\kappa\lambda\mid \alpha \in x\} \in \mathcal{U}$. $\mathcal{U}$ is \emph{normal} if it is fine and for every choice function $f \colon A \to \lambda$, where $A \in \mathcal{U}$, there is an $\mathcal{U}$-large set $B \subseteq A$, such that $f\restriction B$ is constant.
\end{definition}
\begin{lemma}[Reinhardt, Magidor]
There is a normal measure on $P_\kappa\lambda$ if and only if there is an elementary embedding $j \colon V \to M$, with critical point $\kappa$ such that $M$ is closed under $\lambda$-sequences and $j(\kappa) > \lambda$.
\end{lemma}
Let us recall that given an elementary embedding $j \colon V \to M$, such that $\crit j = \kappa$, $j(\kappa) > \lambda$ and ${}^\lambda M \subseteq M$, the ultrafilter $\mathcal{U} = \{X \subseteq P_\kappa \lambda \mid j \image \lambda \in j(X)\}$ is a normal measure on $P_\kappa \lambda$. In other words, the \emph{seed} of $\mathcal{U}$ is $j\image \lambda$. 

\begin{remark}
Throughout this section we will always assume that $\kappa$ carries a sequence of normal measures $\mathcal{V}_n$ on $P_\kappa \kappa^{+n}$. Except for a couple of changes in the notations, there is no harm in reducing the hypothesis to $\mathcal{V}_n$ being $\kappa$-complete fine measure on $P_\kappa \kappa^{+n}$.  
\end{remark}
We would like to define an iteration using this sequence of supercompact measures. 

\[\begin{matrix}
V &=& M_0 & j_{0,0} = id\\  
M_{n + 1} &=& \Ult(M_n, j_n(\mathcal{V}_n)), & j_{n,n+1} \colon M_n \to M_{n+1}\\ 
&\vdots&\\ 
M_\omega & =& \lim \langle M_n, j_{n,m} \mid n \leq m < \omega\rangle  & j_{n,\omega} \colon M_n \to M_\omega 
\end{matrix}\]
Either Gaifman's arguments, or the modern version using the completeness of the measures show that:
\begin{lemma}
$M_\omega$ is well founded.
\end{lemma}
As usual, we let $j_n = j_{0,n}$ for $n \leq \omega$.
 
Unlike the case of the Prikry forcing, this time the seeds of the embeddings are moved by the later steps of the iterations. Thus, we need to define the sequence $P$ in a more precise way.

Let \[p_n = j_{n+1, \omega} \left(j_{n,n+1}\image j_n(\kappa^{+n})\right) = j_{n,\omega}\image j_n(\kappa^{+n})\]
and let $P = \langle p_n \mid n < \omega\rangle$\footnote{Here, if we assume that $\mathcal{V}_n$ are merely $\kappa$-complete fine ultrafilters, instead of normal ones, we will need to change the definition of $p_n$ to be $j_{n+1, \omega} \left([id]_{j_n(\mathcal{V}_n)}\right)$.}.

Let us collect a couple of useful properties of the iteration.
\begin{lemma}\label{lemma:critical-points-GS}
For every $n < m \leq \omega$, $\crit j_{n,m} = j_n(\kappa)$. 
\end{lemma}
\begin{lemma}\label{lemma:representing-elements-GS}
For every $x \in M_n$ there is a function $f \colon \prod_{m < n} P_\kappa \kappa^{+m} \to V$ such that 
\[x = j_n(f)(\bar{p}_0,\dots, \bar{p}_{n-1}),\]
where $\bar p_m = j_{n,m}\image j_m(\kappa^{+m})$.
In particular, for every $x \in M_\omega$ there is a natural number $n$ and $f \colon \prod_{m < n} P_\kappa \kappa^{+m} \to V$ such that $x = j_\omega(f)(p_0,\dots, p_{n-1})$.
\end{lemma}
\begin{lemma}
Let $\lambda$ be a regular cardinal in $M_n$,  $\lambda \notin [j_n(\kappa), j_n(\kappa^{+\omega}))$. Then, $\lambda$ is a continuity point of $j_{n,\omega}$.
\end{lemma}
\begin{proof}
This is a corollary of Lemma \ref{lemma:critical-points-GS} and Lemma \ref{lemma:representing-elements-GS}. Let us prove the lemma for $n = 0$. The general case is similar, with slightly more complicated notations. If $\lambda < \kappa$ then $\lambda$ is a fixed point of $j_{\omega}$. If $\lambda > \kappa^{\omega}$, then let $\alpha < j_{\omega}(\lambda)$. So, $\alpha$ is represented by some function $f$ as above. But, $|\dom f| \leq \kappa^{+n} < \lambda$, so $\sup \range f  = \gamma < \lambda$, and therefore $\alpha < j_{\omega}(\gamma) < j_\omega(\lambda)$. 
\end{proof}
So, the proof of Lemma \ref{lemma:preserving-cardinals-Prikry} goes throughout unchanged and we conclude that every $M_{\omega}$-regular cardinal $\lambda$ which is not in the interval $[j_\omega(\kappa), j_\omega(\kappa^{+\omega}))$ remains regular in $M_\omega[P]$. 
\begin{lemma}
In $M_\omega[P]$, $\forall n < \omega,\,\cf j_\omega(\kappa^{+n}) = \omega$.
\end{lemma}
\begin{proof}
Look at $\alpha_m = \sup j_{m,\omega} \image j_m(\kappa^{+n})$. For every $m \geq n$, this is an ordinal below $j_\omega(\kappa^n)$, as $j_m(\kappa^{+n})$ is a discontinuity point of $j_{m,m+1}$. But, for every ordinal $\gamma < j_\omega(\kappa^{+n})$ there is some large enough $m$ such that $\gamma \in \range j_{m,\omega}$ and thus $\gamma < \alpha_m$. Therefore, $\sup_{m \geq n} \alpha_m = j_\omega(\kappa^{+n})$. 

The sequence $\langle \alpha_m \mid m <\omega\rangle$ can be computed from $P$: $\alpha_m = \sup \left(p_m \cap j_\omega(\kappa^{+n})\right)$.
\end{proof}
We conclude that the successor of $j_\omega(\kappa)$ in $M_\omega[P]$ is $j_\omega(\kappa^{+\omega+1})$.  

By Remark \ref{remark:chain-condition-from-width}, $M_\omega[P]$ is $j_\omega(\kappa^{+\omega+1})$-c.c.\ generic extension of $M_\omega$.

We are now ready for the proof of the Bukovsk\'y-Dehornoy Theorem for the Gitik-Sharon forcing. The proof is essentially the same as the one for the Prikry forcing, so we will sketch it.
\begin{theorem}\label{thm:BD-for-Gitik-Sharon}
$M_\omega[P] = \bigcap M_n$.
\end{theorem}
\begin{proof}
First, we need to show that $M_\omega[P]$ is closed under countable sequences of ordinals. Indeed, if $\langle \alpha_n \mid n < \omega\rangle$ is a countable sequence of ordinals, then for each $n$, pick a function $g_n \colon \prod_{m < m_n} P_\kappa \kappa^{+m}$ such that
\[\alpha_n = j_\omega(g_n)(p_0,\dots, p_{m_n-1}),\]
using Lemma \ref{lemma:representing-elements-GS}. Then, since $j_{\omega}(\langle g_n \mid n < \omega\rangle) = \langle j_\omega(g_n) \mid n < \omega\rangle \in M_\omega$, we get that:
\[\langle \alpha_n \mid n < \omega\rangle = \langle j_\omega(g_n)(P \restriction m_n) \mid n < \omega\rangle \in M_\omega[P].\]

Next, given a set of ordinals $X \in \bigcap M_n$, let us define the sets:
\[Y_n = \{\alpha \mid j_{n,\omega}(\alpha) \in X\}\quad, Z_n = j_{n,\omega}(Y_n)\]
and verify that $\alpha \in X$ iff for all large $n$, $\alpha \in Z_n$.
\end{proof}

One of the interesting features of the Gitik-Sharon forcing relates to the behavior of $\kappa^{+\omega+1}$. In particular, if $\kappa^{+\omega}$ is strong limit cardinal then it is a fixed point of $j_{\omega}$. Thus, under suitable cardinal arithmetic assumption, the Gitik-Sharon forcing cannot introduce certain sufficiently absolute objects.    

Let us start with a simple observation.
\begin{claim}\label{claim:special-tree-in-GS}
Let us assume that $\kappa^{+\omega}$ is a strong limit cardinal. 

There is no special Aronszajn tree on $\kappa^{+\omega+1}$ in $V$ if and only if there is no special Aronszajn tree on the successor of $j_\omega(\kappa)$ in $M_\omega[P]$.
\end{claim}
\begin{proof}
Let us assume that there is such a tree $T$ in $M_\omega[P]$. So, this is a tree of height $\left(j_\omega(\kappa)^{+}\right)^{M_\omega[P]} = \kappa^{+\omega+1}$ such that there is a function $f\colon T \to j_\omega(\kappa)$ which is injective on chains. Since $M_\omega[P] \subseteq V$, $T, f \in V$ and since $|j_\omega(\kappa)|^V = \kappa^{+\omega}$, $T$ is a special Aronszajn tree in $V$. 

The other direction is similar: if $T$ is a special Aronszajn tree in $V$, then $j_\omega(T)$ is a special Aronszajn tree in $M_\omega$. But being special is upwards absolute, so it is special in $M_\omega[P]$ as well. 
\end{proof}

Theorem \ref{thm:BD-for-Gitik-Sharon} implies that $M_\omega[P]$ is closed under $\kappa$-sequences with respect to $V$, since it is an intersection of models which are closed under $\kappa$-sequences of ordinals. The following claim (and its variant) was used in \cite{EskewHayut2021} in order to derive instances of Chang's Conjecture in the Gitik-Sharon extension.
\begin{claim}\label{claim:cc-in-GS}
Let $\kappa^{+\omega}$ be a strong limit. Let us assume that Chang's Conjecture $(\kappa^{+\omega+1}, \kappa^{+\omega})\chang (\rho^+, \rho)$ holds for some $\rho < \kappa$. 

Then $M_\omega[P] \models (j_\omega(\kappa)^+, j_\omega(\kappa))\chang (\rho^+,\rho)$.
\end{claim}
\begin{proof}
Let $\mathcal{A}$ be an algebra in $M_\omega[P]$ on $j_{\omega}(\kappa)^+ = \kappa^{+\omega+1}$. Then, by Chang's Conjecture in $V$, there is $\mathcal{B} \prec \mathcal{A}$ of order type $\rho^+$. But $\mathcal{B} \in M_\omega[P]$, by the closure of the model. 
\end{proof}
Let us remark that the hypothesis of the claim follows from the assumption that $\kappa$ is $\kappa^{+\omega+1}$-supercompact.

Claim \ref{claim:cc-in-GS} fails if $\kappa^{+\omega}$ is not a strong limit. In this case, the Gitik-Sharon forcing adds a good scale (see Definition \ref{def:good-scale}), which implies that every instance of Chang's Conjecture as in the claim does not hold. 
\subsection{Scales}
This subsection deals with the applications of the Bukovsk\'y-Dehornoy Theorem to the behavior of scales at a Prikry type extension. This type of application was used in \cite{BenNeriaHayutUnger}, and the basic ideas are taken from there. Here we apply it for the Gitik-Sharon forcing, in order to derive the preservation of a bad scale in the generic extension by the Gitik-Sharon forcing.

The applications in this subsection used only the fact that the generic extension $M_\omega[P]$ is closed under countable sequences.

Scales are one of the basic objects in the Shelah's PCF theory, see \cite{ShelahCardinalArithmetic, AbrahamMagidor}. Let us present a special case.
\begin{definition}[Shelah]
Let $\lambda$ be a singular cardinal of countable cofinality and let $\langle \lambda_n \mid n < \omega\rangle$ be a sequence of regular cardinals converging to $\lambda$. A \emph{scale} on $\prod \lambda_n$ is a sequence of functions $\langle f_\alpha \mid \alpha < \mu\rangle$, $f_\alpha \in \prod \lambda_n$ such that:
\begin{itemize}
\item For all $\alpha < \beta$, $\{n \mid f_\alpha(n) \leq f_\beta(n)\}$ is co-finite. 
\item For all $g \in \prod \lambda_n$, there is $\alpha < \mu$ such that $\{n \mid g(n) \leq f_\alpha(n)\}$ is co-finite.
\end{itemize}
\end{definition}
We will denote the assertion "$\{n \mid f(n) \leq g(n)\}$ is co-finite" by $f \leq^* g$ and we call the set $n$ of values such that $f(n) > g(n)$, the set of violations of the inequality. 

It is clear that every product carries a scale of some length. 
\begin{lemma}[Shelah]
The lengths of any two scales on the same product have the same cofinality.
\end{lemma}
In particular, it makes sense to assume that the length of the scale is always a regular cardinal. 
\begin{theorem}[Shelah]
For every singular cardinal $\lambda$ of countable cofinality, there is a sequence $\langle \lambda_n \mid n < \omega\rangle$ of regular cardinals and a scale $\langle f_\alpha \mid \alpha < \lambda^+\rangle$ on $\prod \lambda_n$.
\end{theorem}
The notion of scales is much wider than this limited definition. For our purposes, we would like to focus on scales of minimal length, and consider ones with a better behaviour.
\begin{definition}\label{def:good-scale}
Let $\mathcal S = \langle f_\alpha \mid \alpha < \lambda^+\rangle$ be scale. An ordinal $\beta < \lambda^+$ of uncountable cofinality is \emph{good} if there is $n < \omega$ and $A \subseteq \beta$ cofinal such that for every $\alpha_0 < \alpha_1$ in $A$, $\{m \mid f_{\alpha_0}(m) > f_{\alpha_1}(m)\}\subseteq n$. 

An ordinal $\beta$ is \emph{bad}, if it is not good.

A scale $\mathcal{S}$ is a \emph{good} if there are club many good points. A scale is \emph{bad} if it has stationarily many bad ordinals. 
\end{definition} 
The notion of bad (and good) points in a scale can be traced back to \cite{ShelahBadPoints}. The connections between the notions of good scales to other anti-compactness principles such as square principles were summarized and investigated, for example, in the seminal paper \cite{CummingsForemanMagidor}.

\begin{remark}[Shelah]\label{remark:same-bad-points}
If $\mathcal{S}, \mathcal{S}'$ are both scales on the same product of length $\lambda^+$, then their sets of good ordinals agree up to a non-stationary error.
\end{remark}
\begin{proof}
Let $\mathcal{S} = \langle f_\alpha \mid \alpha < \lambda^+\rangle, \mathcal{S}' = \langle f'_\alpha \mid \alpha < \lambda^+\rangle$. Let $C$ be the club of all ordinals $\delta$ such that for all $\alpha < \delta$, there is $\alpha' < \delta$ such that $f_\alpha \leq^* f'_{\alpha'}$ and vise verse. Then, an ordinal $\delta \in C$ is good in $\mathcal{S}$ if and only if it is good in $\mathcal{S}'$: Take $A \subseteq \delta$ witnessing $\delta$ being good in $\mathcal{S}$. Then, by induction, pick a sequence of ordinals $\alpha_i < \alpha_i' < \alpha_{i+1} < \cdots$ such that $\alpha_i \in A$ for all $i$ and $f_{\alpha_i} \leq^* f'_{\alpha'_i} \leq^* f_{\alpha_{i+1}}$. 

For every $i$, there is a bound for the error in the inequality $f_{\alpha_i} \leq^* f_{\alpha_i'}$ and $f'_{\alpha_i'} \leq^* f_{\alpha_{i+1}}$, $n_i$. As $\cf \delta > \omega$, there is some $n_*$ such that $A' = \{\alpha_i' \mid n_i = n_*\}$ is unbounded. So, $A'$ witnesses $\delta$ being good from $\mathcal{S}'$.   
\end{proof}
In the paper \cite{GitikSharon2008}, Gitik and Sharon solved Woodin's question whether it is consistent that in a successor of a singular, $\SCH$ fails but there is no weak square. They achieved that by showing that in the model obtained by forcing with the Gitik-Sharon forcing, starting with a supercompact cardinal $\kappa$ such that $2^\kappa > \kappa^{+\omega+1}$, there are bad scales on the successor of $\kappa$. In our framework, this theorem translates to the following statement about $M_\omega[P]$.
\begin{theorem}[Gitik-Sharon]
Assume that in $V$ there is a scale on $\prod \kappa^{+n}$ in which the set of bad points $S$ contains stationarily many ordinals of cofinality $<\kappa$. Then, there is a bad scale in the successor of $j_\omega(\kappa)$ in $M_\omega[P]$. 
\end{theorem}
\begin{proof}
Following the ideas from Claims \ref{claim:special-tree-in-GS} and \ref{claim:cc-in-GS}, we would like to somehow use the scale from $V$ in order to verify that a corresponding scale from $M_\omega[P]$ cannot be good.

Let us note that if $\kappa^{+\omega}$ is a strong limit, this is rather trivial, but we are mostly interested in the case that $2^\kappa$ is large. In this case, $\langle \kappa^{+n} \mid n < \omega\rangle$ is not necessarily cofinal at $j_\omega(\kappa)$, as even $j_1(\kappa)$ might be larger than all of those cardinals. Thus, we study the product of $\mu_n = \left(j_n(\kappa^{+n + 1})\right)^{M_\omega}$.
\begin{claim}
For every $n < \omega$, $\kappa^{+n+1}$ is a continuity point of $j_n$.
\end{claim}
Note that $\mu_n$ is a regular cardinal in $M_n$, and by absoluteness, it is regular in $M_\omega[P]$ as well. Moreover, since $j_{n+1}(\kappa) > j_n(\kappa^{+n+1})$, $\sup \mu_n = j_\omega(\kappa)$. 

Pick a scale on $\prod \mu_n$, $\mathcal T = \langle g_\alpha \mid \alpha < \mu_*\rangle$. Even though $\mu_n$ might not be regular in $V$, the two defining properties of a scale (weakly increasing and cofinal) holds in $V$ for $\mathcal{T}$ since $M_\omega[P]$ is closed under $\omega$-sequences or ordinals.

We would like to collapse $\mathcal T$ to a scale in $V$. Recall that for all $n$, $\sup j_n\image \kappa^{+n+1} = \mu_n$. Let us define for every $\alpha$, 
\[h_\alpha(n) = \min \{\beta < \kappa^{+n+1} \mid g_\alpha(n) < j_n(\beta)\},\]
and let $\bar{\mathcal T}$ be $\langle h_\alpha \mid \alpha < \mu_*\rangle$.  
\begin{claim}
$\bar{\mathcal T}$ is a scale in $V$ on $\prod \kappa^{+n+1}$. 
\end{claim}
\begin{proof}
Let $\alpha < \beta$. Since $g_\alpha(n) \leq g_\beta(n)$ for almost all $n$, we conclude that $h_\alpha(n) \leq h_\beta(\alpha)$ for almost all $n$. 

Let $h$ be arbitrary. Then, let us look at the $\tilde{h}(n) = j_n(h(n))$. Since $M_\omega[P]$ is closed under countable sequences, $\tilde{h} \in M_\omega[P]$. Therefore, there is $\alpha$ such that $\tilde{h}(n) \leq g(\alpha)$ for almost all $n$. We conclude that $h(n) \leq h_\alpha(n)$ for almost all $n$.\end{proof}
In particular, $\cf^V \mu_* = \kappa^{+\omega+1}$. Fix in $V$ an arbitrary continuous sequence of ordinals $\vec\beta=\langle \beta_i \mid i < \kappa^{+\omega+1}\rangle$ cofinal at $\mu_*$, and let $\bar{\bar{\mathcal{T}}} = \langle h_{\beta_i} \mid i < \kappa^{+\omega+1}\rangle$. Fix a scale $\mathcal{S}$ as in the hypothesis of the theorem. Let $C$ be the club from the proof of Remark \ref{remark:same-bad-points} for the scales $\mathcal{S}, \bar{\bar{\mathcal{T}}}$. 
\begin{claim}
For every $\alpha \in C$ of cofinality $<\kappa$, if $\alpha$ is bad for $\mathcal{S}$ in $V$ then $\beta_\alpha$ it is bad for $\mathcal{T}$ in $M_\omega[P]$. 
\end{claim}
\begin{proof}
Assume that this is not the case. Since $\cf \alpha^V < \kappa$ and the sequence $\vec\beta$ is continuous, $\cf^V \beta_\alpha < \kappa$. By the closure of $M_\omega[P]$, $\cf^{M_\omega[P]}\beta_\alpha < \kappa$ and moreover, there is $B \subseteq \beta_\alpha$ in $M_\omega[P]$, cofinal and contained in $\{\beta_i \mid i < \alpha\}$. Pick $A \subseteq B$ witnessing $\beta_\alpha$ being good for $\mathcal{T}$. 

Then, $A$ witnesses that $\beta_\alpha$ is good for $\bar{\mathcal T}$, which means that $\bar{A} = \{i \mid \beta_i \in A\}$ witness that $\alpha$ is good for $\bar{\bar{\mathcal{T}}}$  
and thus by Remark \ref{remark:same-bad-points}, $\alpha$ is good from $\mathcal{S}$.  
\end{proof}  
Finally, if $D$ is a club in $M_\omega[P]$ then $D\in V$ and it is a closed and unbounded subset of the ordinal $\mu_*$. We conclude that $\{\alpha \mid \beta_\alpha \in D\}\in V$ is a club in $\kappa^{+\omega+1}$. Together with the previous claim, we see that the set of bad points in $\mathcal{T}$ has to be stationary in $M_\omega[P]$.
\end{proof}
The standard proofs of properties of scales in Prikry type extensions need some bounding lemmas. Looking closely into the arguments, one can identify the parallel parts: the bounding lemmas correspond to the construction of the scale $\bar{\bar{\mathcal{T}}}$ from $\mathcal{T}$. Nevertheless, using the Bukovsk\'y-Dehornoy method, we do not need to talk about names and use the strong Prikry Property is order to partially realize them. 

Finally, we can prove:
\begin{theorem}[Gitik-Sharon]
Let $\kappa$ be a supercompact cardinal such that $2^\kappa \geq \kappa^{+\omega+2}$. Then, there is a generic extension in which $\kappa$ is a strong limit singular, there is a bad scale on $\kappa^{+}$ and $2^\kappa > \kappa^+$.
\end{theorem}
\begin{proof}
In $M_\omega$, there is a forcing notion adding $P$, by Bukovsk\'y Theorem, and since $M_\omega[P]$ satisfies the conclusion of the theorem (the failure of $\SCH$ and the existence of a bad scale), by the forcing theorem, there is a condition forcing that. 

By elementarity, the same holds in $V$: there is a forcing notion and a condition in it forcing the failure of $\SCH$ together with the existence of a bad scale on $\kappa$.  
\end{proof}

By combining interleaved collapses in the Gitik-Sharon forcing, one can obtain a model in which $\SCH$ fails at $\aleph_{\omega^2}$ and there is a bad scale on $\aleph_{\omega^2}$. In the next section we will address the issue of adding collapses to a forcing notion in which the Bukovsk\'y-Dehornoy Theorem holds, thus allowing us to obtain the full result.  
\section{Combining iterated ultrapowers with forcings}\label{section:EBPF}
In this section, we will give a couple of examples for extensions of an iterated ultrapower, $M_\omega$, using an object which can be obtained only in a generic extension of $V$. While an additional level of complexity is added to the whole process, still a few key components are preserved. Our goal model can be presented as the intersection of a (definable) decreasing sequence of models in a generic extension, so many of the arguments from the previous sections will be applicable here as well. 

The most notable change is that we are losing the elementary embeddings between the models in the chain. This makes the proof of the parallel intersection theorems more involved.   

We will deal with two main cases: the Extender Based Prikry Forcing and adding interleaved collapses.
\subsection{Extender Based Prikry forcing}\label{subsection:BD-for-EBPF}
In this section, we follow closely results and ideas from Merimovich, \cite{Merimovich2007, Merimovich2021} in order to derive a Bukovsk\'y-Dehornoy theorem for extender based Prikry forcing.
 
There are several definitions for \emph{extenders} in the literature, see for example \cite{GitikHandbook,Kanamori}. For our purposes, a $(\kappa,\lambda)$-extender $E$ is a combinatorial (set) object coding an elementary embedding $j \colon V \to M$ with $\crit j = \kappa$, $j(\kappa) > \lambda$, $M$ is closed under sequences of length $\kappa$ and for every $x \in M$ there are $f \colon \kappa \to V$ in $V$ and $\gamma < \lambda$ such that $j(f)(\gamma) = x$. In particular, the width of the embedding $j$ is $\leq\kappa$. 

\begin{lemma}[$\kappa$-directness]\label{lemma:kappa-directness}
Let $E$ be a $(\kappa,\lambda)$-extender and $j\colon V \to M$ be the derived elementary embedding.
 
For every $A\subseteq \lambda$, $|A| \leq \kappa$ there is $\gamma < \lambda$ such that for every $\delta \in A$ there is $f_\delta \colon \kappa \to \kappa$ such that $j(f_\delta)(\gamma) = \delta$. 
\end{lemma}
\begin{proof}
Since $M$ is closed under $\kappa$-sequences, $A\in M$ and moreover, some enumeration $\langle \delta_i \mid i < \kappa\rangle \in M$. Thus, there is $\gamma < \lambda$ and $g \colon \kappa \to V$ such that $j(g)(\gamma) = \langle \delta_i \mid i < \kappa\rangle$. 

So, define $f_{\delta_i}(\zeta) = g(\zeta)(i)$ and the result follows from elementarity.
\end{proof}

As a definable elementary embedding, the elementary embedding derived from an extender can be iterated and the direct limit of such an iteration is well founded. 

Fix a $(\kappa, \lambda)$-extender $E$, and let $\langle M_n, j_{n,m} \mid n \leq m \leq \omega\rangle$ be the corresponding iteration of $E$. 

\begin{lemma}\label{lemma:representing-elements-EBPF-Momega}
For every $x \in M_\omega$ there are $n < \omega$, $f \colon \kappa^{n} \to V$ and $\gamma_0, \dots, \gamma_{n-1}$, such that $\gamma_k < j_k(\lambda) < j_{k+1}(\kappa)$ and $x = j_\omega(f)(\gamma_0,\dots, \gamma_{n-1})$. 
\end{lemma}

\begin{lemma}\label{lemma:width-of-iteration}
The width of the embedding $j_{n,m}$ for $n < m \leq \omega$ is $\leq j_{n}(\kappa)$.
\end{lemma}
\begin{proof}
First, for $m = n+1$, this follows from our initial hypothesis on the extender $E$, using elementarity.

Let us assume that the width of the embedding $j_{n,m}$ is $\leq j_n(\kappa)$. Let $x \in M_{m+1}$. Since the width of $j_{m,m+1}$ is $j_m(\kappa)$, there is $y \in M_m$ such that $|y| \leq j_m(\kappa)$ and $x \in j_{m,m+1}(y)$. 

Since the width of $j_{n,m}$ is $j_n(\kappa)$, we conclude that there is $z \in M_n$ such that $|z| \leq j_n(\kappa)$ and $y \in j_{n,m}(z)$. By elementarity, taking $z' = \{y' \in z \mid |y'| \leq j_n(\kappa)\}$ we conclude that $y \in j_{n,m}(z')$. Finally, let $w = \bigcup z'$ --- this set is a union of $\leq j_n(\kappa)$ many sets of cardinality $\leq j_n(\kappa)$, so $|w| \leq j_n(\kappa)$. 

So, $x \in j_{n,m+1}(w)$ since $x \in j_{m,m+1}(y)$ and $j_{m,m+1}(y) \in j_{m,m+1}(j_{n,m}(z))$.

So, by induction, the claim holds for all $n < m < \omega$. 

Now, let us deal with the width of $j_{n,\omega}$. If $x \in M_\omega$, then there is $m > n$ (without loss of generality) and $\bar x \in M_m$ such that $x = j_{m,\omega}(\bar x)$. Since $j_{n,m}$ has width $j_n(\kappa)$, there is $y \in M_n$ such that $\bar x \in j_{n,m}(y)$. 

So, $x \in j_{n,\infty}(y)$, as wanted. 
\end{proof}

Merimovich proved that one can obtain a generic for the extender based Prikry forcing (defined in a proper way), but forcing with the direct extension order taking an iterated ultrapower and adding the generator. The following theorem works the details for the Bukovsk\'y-Dehornoy intersection theorem for this forcing, again without explicitly defining the forcing notion.
\begin{theorem}\label{thm:BD-for-EBPF}
Let $H \subseteq \mathbb{P}^*$, be a $V$-generic filter for the forcing $\mathbb{P}^*=\{f \colon \lambda \to \kappa^{<\omega} \mid |f| \leq \kappa\}$, ordered by reverse inclusion. 

Let us define inductively
\[H_{n+1} = \text{upwards closure of } \{j_{n,n+1}(p) ^\smallfrown \{\langle j_{n,n+1}(\alpha),\alpha\rangle \mid \alpha \in \dom p\} \mid p \in H_n\}\]

Let $G \colon j_{\omega}(\lambda) \times \omega \to j_\omega(\kappa)$ be defined by:
\[\begin{matrix}
G(\alpha, n) = \gamma \iff & & \\ 
\exists m < \omega, & p \in H_m, & \alpha \in \dom j_{m,\omega}(p), &  n \in \dom (j_{m,\omega}(p)(\alpha)) \\ 
\gamma = j_{m,\omega}(p)(\alpha)(n) 

\end{matrix}
\] 
Then $M_\omega[G] = \bigcap M_n[H_n]$. Moreover, in $M_\omega[G]$ $\cf j_\omega(\kappa) = \omega$ and $2^{j_\omega(\kappa)} = j_\omega(\lambda)$.
\end{theorem}
\begin{proof}
First, the verification that $\cf^{M_\omega[G]} j_\omega(\kappa) = \omega$ and $G$ forms a scale of length $j_\omega(\lambda)$ on $j_\omega(\kappa)$ is straight-forward. Let us focus in the proof of the intersection theorem.

As in Theorem \ref{thm:BD-for-vanilla}, we need to show first that the model is closed under countable sequences.
\begin{lemma}\label{lemma:EBPF-sigma-closed}
$M_\omega[G]$ is closed under $\omega$-sequences.
\end{lemma}
\begin{proof}
Since $M_\omega[G]$ is a model of $\ZFC$, it is enough to prove the claim for sequences of ordinals. 

The crux of the argument is \cite[Corollary 2.6]{Merimovich2007}.

Fix $\alpha$ an ordinal in $M_\omega[G]$. So, there is a function $f \colon \kappa^n \to \Ord$ and $\gamma_0,\dots, \gamma_{n-1}$ as in Lemma \ref{lemma:representing-elements-EBPF-Momega}, such that $j_{\omega}(f)(\gamma_0,\dots, \gamma_{n-1}) = \alpha$.

Work in $V$. Fix an elementary substructure $N \prec H(\chi)$ for some large $\chi$, such that $\kappa \subseteq N$, $|N| = \kappa$, $E, \alpha \in N$. Let us apply the $\kappa$-directness of $E$, in the sense of Lemma \ref{lemma:kappa-directness}, and obtain an ordinal $\rho$ such that for every $\beta \in N$ there is $h$ such that $j(h)(\rho) = \beta$. In particular, this is true for $\gamma_0$, so there is $h_0$ such that $j(h_0)(\rho) = \gamma_0$. As $N$ knows that the width of $j$ is $\kappa$, there is $a \subseteq \lambda$ in $N$ of cardinality $\kappa$ such that $\gamma_1 \in j(a)$. In particular, since $a \subseteq N$, there is $h_1$ such that $j_2(h_1)(j_1(\rho)) = \gamma_1$. Let us claim that we can "trace back" $h_1$ to a function in $V$. Indeed, in $N$ there is $\bar\rho < \lambda$ which is Rudin-Keisler above $a$. In particular, there is $\bar{h}_2 \in M_1^N$ such that $j_{1,2}(\bar h_2)(\bar \rho) = \gamma_2$. But, $\bar{h}_2 = j_1(\bar{\bar{h}}_2)(\bar\gamma_2)$ for some $\bar\gamma_2 \in N$ and thus we conclude that all those computations can be made using only $\rho$.  

Continue this way, we conclude that there is a function $h \colon \kappa^n \to \Ord$ such that $j_\omega(h)(\rho, j_1(\rho),j_2(\rho),\dots, j_{n-1}(\rho)) = \alpha$.

Now, let $\langle\alpha_n \mid n < \omega\rangle$ be a sequence of ordinals. Applying the above arguments with $N$ containing the sequence $\langle \alpha_n \mid n < \omega\rangle$ instead of just a single ordinal $\alpha$, we obtain $\rho$ and a sequence of functions $\langle h_n \mid n < \omega\rangle$ such that for all $n$, $j_n(h_n)(\rho,j_1(\rho),\dots, j_{m_n-1}(\rho)) = \alpha_n$. 

Since the sequence of images of $\rho$ is a final segment of $G(j_\omega(\rho))$, we conclude that $\langle \alpha_n \mid n < \omega\rangle \in M_\omega[G]$. 
\end{proof}

Let $X \in \bigcap M_n[H_n]$ be a set of ordinals. For each $n$, let us pick a $j_n(\mathbb{P}^*)$-name $\dot \tau_n$ for 
\[Y_n = \{\alpha \mid j_{n,\omega}(\alpha) \in X\}.\]
So, $\dot\tau_n^{H_n} = Y_n$. 

Let $\sigma_n = j_{n,\omega}(\tau_n) \in M_\omega$. Since $M_\omega[G]$ is closed under $\omega$-sequences, the sequence of names $\langle \sigma_n \mid n < \omega\rangle$ is a member of $M_\omega[G]$.

The problem is the we do not have access to the actual generics $H_n$ from which the names were realized. In order to overcome this, we will need to isolate a relevant version of the Prikry forcing that hides inside our forcing and show that its generic is unique (up to shifts). Here we must diverge from the thesis of the paper and work with forcing notion. As we would like to avoid cluttering this part of the proof with definitions, we refer the reader to \cite[Section 1.2]{GitikHandbook} for the definition of Prikry forcing on trees. 

Let $s \subseteq j_{\omega}(\lambda)$ be a set of size $\leq \kappa_\omega$ in $M_\omega$, such that $s$ is the intersection with $\lambda$ of an elementary submodel of $H^{M_\omega}(\chi)$. Let $U(s)$ be the corresponding measure: $A \in U(s) \iff \{\langle k(\alpha),\alpha\rangle\mid \alpha \in A\} \in k(A)$ for $k = j_{j_\omega(E)}$. 

Let $\mathbb{Q}_s$ be the tree Prikry forcing defined using the measure $U(s)$. Clearly, $U(s)$ is the image of some measure of the form $U(\bar s)$ on $M_n$ for some $n < \omega$.
 
\begin{lemma}\label{lemma:EBPF-realizing-names}
There is a sequence $\vec t = \langle t_n \mid n < \omega\rangle$ which is generic for $\mathbb{Q}_s$, and 
a condition $p \in j_{\omega}(\mathbb{P}^*)$, such that $p^\smallfrown \langle t_0, \dots, t_{n-1}\rangle$ are compatible with $G$ for all $n$.  

Moreover, the sequence $\vec t$ is unique, up to a finite shift.
\end{lemma}
\begin{proof}
This follows from the Merimovich's criteria, \cite{Merimovich2021}:

First, in order to get such a sequence pick any $p \in \bigcup j_{n,\omega}\image H_n$ such that $\dom p \supseteq s$. Let $p \in j_{n,\omega}(H_n)$. Take, $t_m$ to be the added coordinates in step $n + m$, namely for $\alpha \in \dom j_{n,m}(\bar{p})$ we let $t_m(j_{m,\omega}(\alpha)) = \alpha$.  

Let $p, p'$ be conditions in $j_\omega(\mathbb{P}^*)$ with domain $s$, such that there are sequences $\langle t_n \mid n < \omega\rangle$, $\langle t_n' \mid n < \omega\rangle$ are generic for $\mathbb{Q}_s$, and for all $\alpha \in s$, 
\[p(\alpha) ^\smallfrown \langle t_n(\alpha) \mid n < \omega\rangle =  
p'(\alpha) ^\smallfrown \langle t'_n(\alpha) \mid n < \omega\rangle =  G(\alpha).\]

Let us show that there is $k$ such that for all large $n$, $t_{k + n} = t'_{k'+n}$ and 
\[p^\smallfrown \langle t_0, \dots, t_{k-1}\rangle = {p'} ^\smallfrown \langle t_0', \dots, t'_{k'-1}\rangle.\]

Pick in $M_\omega$ a function $g$ enumerating $s$. For all large $n$, $g \image t_n(\kappa_\omega) = \dom t_n$. In particular, by comparing the values at $\kappa_\omega$, we obtain the possible value of $k' - k$. Moreover, for all large $n$, 
\[\forall \alpha \in \dom t_n,\, t_n(\alpha) > \max(p(\alpha), p'(\alpha)) \text{ and } t_n(\alpha) < t_{n+1}(\kappa_\omega).\]
This is easily obtained by taking the right $U(s)$-large tree. Therefore, for such $n$-s, the value of $t_n(\alpha)$ is the unique ordinal in $G(\alpha)$ which is between $G(\kappa)_n$ and $G(\kappa)_{n+1}$.
\end{proof}

Given a condition $p$ and a sequence $\vec t$ witnessing the validity of the lemma for $s$, let $p_n = p^\smallfrown \langle t_0, \dots, t_{n-1}\rangle$. We call $\langle p_n \mid n < \omega\rangle$ a Prikry sequence for $s$. The lemma indicates that the Prikry sequence is unique, up to an initial segment. We will assume always that $|p_n(\kappa_\omega)| = n$. This makes the Prikry sequences to agree up to an initial segment, without a shift.

\begin{lemma}
Let $\langle p_n \mid n < \omega\rangle$ be a Prikry sequence for some $s$. Then, for all large $n$, $p_n \in j_{n,\omega}(H_n)$.
\end{lemma}
\begin{proof}
Pick $n$ large enough so that $s = j_{n,\omega}(\bar s)$. Let $q \in H_n$ with domain $\bar{s}$. Then, the canonical sequence of generators, starting with $j_{n,\omega}(q)$ is a Prikry sequence for $s$. In particular, letting $\langle q_m \mid m < \omega\rangle$ be the corresponding sequence (adding dummy conditions in the beginning, if needed), we know that for all large $n$, $q_n = p_n$, as we set the shift using the $\kappa_\omega$ coordinate.

Since $q_n \in H_n$, once we go past the dummy coordinates, the conclusion follows.
\end{proof}
Let us claim that $X$ is the set of all $\alpha$ such that there is a condition $p$ and $\vec t$ compatible with $G$, and for all large $n$, 
\[p_n \Vdash \alpha \in \sigma_n\]
Indeed, if $\alpha \in X$, the existence of such $p$ is clear. Otherwise, since for all large $n$, $p_n$ comes from $H_n$, it cannot force contradictory information.
\end{proof}

It is interesting to try to see where the chain condition proof fails. Indeed, any ordinal can still be captured by a set of size $\kappa$. The problem is both the generic for $\mathbb{P}^*$ and the fact that there is no elementary embedding from $V[H]$ to $M_\omega[G]$. 
\subsection{Interleaved Collapses}\label{subsection:BD-for-IC}
In this section, we will describe a situation in which interleaved collapses can be incorporated into the BD setting. Let us discuss first the simple setting of the vanilla Prikry forcing using a normal measure $U$. Let $j \colon V \to M$ be the ultrapower map using the normal measure. 

The following lemma is well known.
\begin{lemma}\label{lemma:guiding generic}
If $2^\kappa = \kappa^+$, then there is an $M$-generic filter for $\Col(\kappa^+, <j(\kappa))$ in $V$. 
\end{lemma}
\begin{proof}
Let us count the maximal antichains of the forcing $\Col(\kappa^+, <j(\kappa)$ in $M$. By the chain condition and since $j(\kappa)$ is inaccessible in $M$, there are $j(\kappa)$ such antichains in $M$. In $V$, $|j(\kappa)| = |\kappa^\kappa| = \kappa^+$. Let $\vec A = \langle A_\alpha \mid \alpha < \kappa^+\rangle$ be an enumeration of all maximal antichains in $M$, $\vec A \in V$.  

The forcing $\Col(\kappa^+, <j(\kappa))^M$ is $\kappa^+$-closed in $M$, and since $M$ is closed under $\kappa$-sequences, it is $\kappa^+$-closed in $V$ as well. Let us define in $V$ a decreasing sequence of conditions, $\langle p_\alpha \mid \alpha < \kappa^+\rangle$ with the property that $p_\alpha \leq q_\alpha \in A_\alpha$. This can be done, using the closure of the forcing (from the point of view of $V$) at limit steps. 

Let $K$ be the upwards closure of $\langle p_\alpha \mid \alpha < \kappa^+\rangle$. Then  $K$ is $M$-generic filter.
\end{proof}
While it seems like the argument relies on the chain condition of the forcing, a similar argument works for the forcing $\Col(\kappa^+, j(\kappa))$, as the number of dense open sets is still $\kappa^+$ from the point of view of $V$. It does not work if $2^\kappa > \kappa^+$. Yet, by carefully constructing the model, at some cases one can obtain an $M$-generic filter in those cases as well. 

From this point, we will only assume that there is an $M$-generic filter for the collapse $\Col(\kappa^+, <j(\kappa))$, but modifying this to other forcing notions that admit a guiding generic does not change the argument.

Let us consider the direct system of ultrapower embeddings, $j_{n, m} \colon M_n \to M_m$, $m \leq \omega$. Let $K_{n + 1} = j_n(K)$. 

\begin{lemma}
For each $n \leq m$, $K_{n + 1}$ is $M_m$ generic for the forcing $\Col(\kappa_n^+, <\kappa_{n+1})$.
\end{lemma}
\begin{proof}
First, by elementarity, $K_{n + 1}$ is $M_n$-generic. Moreover, since $M_m \subseteq M_n$ and 
\[\Col^{M_n}(\kappa_n^+, <\kappa_{n+1})=\Col^{M_m}(\kappa_n^+, <\kappa_{n+1}),\]
we conclude that $K_{n+1}$ is $M_m$ generic as well. 
\end{proof}
\begin{lemma}
Let $K_0$ be $V$-generic for $\Col(\omega_1, <\kappa)$. 

For each $n$, $K_0 \times K_1 \times \cdots \times K_n$ is $M_n$-generic for $\prod_{i \leq n} \Col(\kappa_{i-1}^+, <\kappa_i)$. 
\end{lemma}
\begin{proof}
Using Easton Lemma: if $H$ is generic for a $\lambda$-c.c.\ forcing and $G$ is generic for a $\lambda$-closed forcing then $H$ and $G$ are mutually generic.
\end{proof}
Let $\vec K = \langle K_n \mid n < \omega\rangle$.
\begin{theorem}\label{thm:BD-for-interleaved-collapses}
$M_\omega[\vec K] = \bigcap M_n[\vec K \restriction n + 1]$.
\end{theorem}
\begin{proof}
As in the proof of Theorem \ref{thm:BD-for-vanilla}, we need to show first closure under $\omega$-sequences of ordinals and then to use a similar (but simpler) argument as in Theorem \ref{thm:BD-for-EBPF} in order to conclude the full theorem.

Let us begin with the closure under $\omega$-sequences.
\begin{lemma}
$M_\omega[\vec K]$ is closed under $\omega$-sequence of ordinals from $V[K_0]$. 
\end{lemma}
\begin{proof}
Indeed, it is easy to extract $P = \langle \kappa_n \mid n < \omega\rangle$ from $\vec K$. Therefore, $M_\omega[\vec K] \supseteq M_\omega[P]$ which is closed under countable sequences of ordinals from $V$, by Lemma \ref{lemma:countable-closure-vanilla-Prikry}. But $V$ and $V[K_0]$ have the same countable sequences, so the conclusion holds.  
\end{proof}
Let $X \in \bigcap M_n[\vec K \restriction n + 1]$ be a set of ordinals. Let us define $Y_n$ and the corresponding name $\dot\tau_n$ as before:
\[Y_n = \{\alpha\mid j_{n,\omega}(\alpha) \in X\}, \quad \dot\tau_n^{\vec K \restriction n + 1} = Y_n.\]
Let us look at $j_{n,\omega}(\dot\tau_n)$, this is a name with respect to the forcing \[\mathbb{Q}_n^j = \left(\prod_{i < n} \Col(\kappa_{i-1}^+, \kappa_i)\right) \times \Col(\kappa_{n-1}^+, \kappa_{\omega}).\] 

\begin{claim}\label{claim:interleaved-collapses-realizing-names}
$\alpha \in X$ if and only if for all large $n$, there is a condition $p \in \vec K \restriction n + 1$ such that (viewing $p$ as a condition in $\mathbb{Q}_n^j$), $p \Vdash_{\mathbb{Q}^j_n} \alpha \in j_{n,\omega}(\dot\tau_n)$.
\end{claim}
\begin{proof}
Indeed, if $\exists \bar\alpha, j_{n,\omega}(\bar \alpha) = \alpha$, then $\bar\alpha \in Y_n$ and there is some $p \in \vec K \restriction n + 1$ that forces that, namely $\bar\alpha\in \dot\tau_n$. Since $j_{n,\omega}(p) = p$, we conclude that $p \Vdash_{\mathbb{Q}^j_n} \alpha \in j_{n,\omega}(\dot\tau_n)$. 

On the other hand, if there is $p \in \vec K \restriction n + 1$ that forces $\alpha \in j_{n,\omega}(\dot\tau_n)$ and $\alpha = j_{n,\omega}(\bar\alpha)$ then again $p = j_{n,\omega}(p)$ and therefore $p \Vdash \bar\alpha \in \dot\tau_n$. 
\end{proof} 

In particular, since $\langle \dot\tau_n \mid n < \omega\rangle \in M_\omega[P]$, we conclude that $X \in M_\omega[\vec K]$.
\end{proof}

From Theorem \ref{thm:BD-for-interleaved-collapses}, it is easy to deduce which cardinals are preserved in $M_\omega[\vec K]$ and even show using Bukovsky's Theorem $M_\omega[\vec K]$ is a generic extension of $M_\omega$ using a $\kappa_\omega^+$-c.c.\ forcing notion. 
\subsection{Extender Based Prikry forcing with interleaved collapses}
Let us combine the results of Subsections \ref{subsection:BD-for-EBPF} and \ref{subsection:BD-for-IC}.

The following well known lemma, shows that it is possible to obtain a guiding generic filter for extender ultrapower. 
\begin{lemma}
Let $E$ be a $(\kappa, \kappa^{++})$-extender and let us assume that $2^\kappa = \kappa^+$. Let $M$ be the extender ultrapower by $E$. 

Then, there is an $M$-generic filter, $K$ for $\Col(\kappa^{+++}, <j(\kappa))$. 
\end{lemma}
\begin{proof}
Let $U$ be the derived normal ultrafilter, and let $i \colon V \to N$ be the normal ultrapower. By Lemma \ref{lemma:guiding generic}, there is an $N$-generic filter $K_0$ for the forcing $\Col((\kappa^{+++})^N, < i(\kappa))$. 

Let $k \colon N \to M$ be the map given by $k([f]_U) = j(f)(\kappa)$. Let $K$ be the upwards closure of $k \image K_0$. 
\begin{claim}
$K$ is $M$-generic filter.
\end{claim}
\begin{proof}
For every $x \in M$, there is $g \colon (\kappa^{++})^N \to N$ such that $x \in \range k(g)$. This is clear, by noting that if $x = j(r)(a)$ for some generator $a$, then without loss of generality, $a \in \kappa^{++}$ and therefore $x \in \range k(i(r) \restriction (\kappa^{++})^N)$.

Given a dense open set $D \in M$, let $g \colon (\kappa^{++})^N \to N$ cover it, and let $D' = \bigcap \{g(\alpha) \mid g(\alpha) \text{ dense open}\}$. By the distributivity of the forcing in $N$, $D'$ is dense open and therefore there is a condition in $K_0$ meeting it.
\end{proof}
\end{proof}

Let us consider now the iteration $\langle j_{n,m} \colon M_n\to M_m\mid n \leq m \leq \omega\rangle$ given by iterating the extender embedding. Let $K_0$ be a $V$-generic for $\Col(\omega_1, <\kappa)$, and let $H_0\subseteq \mathbb{P}^*$ be $V$-generic. 

Let us define $H_n$ and $G$ in the very same way as in Theorem \ref{thm:BD-for-EBPF} and $K_n$ in the same way as in Theorem \ref{thm:BD-for-interleaved-collapses}. 

\begin{theorem}
$\bigcap M_n[\vec K \restriction n + 1][H_n] = M_\omega[\vec K][G]$.
\end{theorem} 
\begin{proof}
As the ideas of the proof are very similar to the proofs of Theorem \ref{subsection:BD-for-EBPF} and Theorem \ref{thm:BD-for-interleaved-collapses} let us sketch the proof.

First, we show that the model $M_\omega[\vec K][G]$ is closed under $\omega$-sequences, using Lemma \ref{lemma:EBPF-sigma-closed}. Then, we obtain a sequence of names $\langle \sigma_n \mid n < \omega\rangle$ as above. Using Lemma \ref{lemma:EBPF-realizing-names}, we obtain local approximations of the filters $j_{n,\omega}\image H_n$ and by the same argument of Claim \ref{claim:interleaved-collapses-realizing-names} we obtain the relevant conditions for the $K$-parts. 
\end{proof}
\section{Magidor and Radin forcing}\label{section:Magidor-Radin}
In \cite{Magidor78}, Magidor introduces a variant of the Prikry forcing that enables one to change the cofinality of a measurable cardinal to be uncountable. This forcing was revised by Mitchell and generalized by Radin. The version that we present here follows Mitchell's definition.\footnote{A variant of this presentation appeared in the unpublished book of Cummings and Woodin.}  As in the other parts of the paper, we are not going to define a forcing notion but rather an extension of an iterated ultrapower.

\begin{definition}
Let $U, U'$ be two normal measures on $\kappa$. Then $U$ is below $U'$ in the Mitchell order if $U \in \Ult(V, U')$.
\end{definition}

\begin{definition}[Mitchell]
Let $o^\mathcal{U}$ be a function from $\kappa + 1$ to ordinals.

A sequence $\mathcal{U} = \langle U_{\alpha, \beta} \mid \alpha \leq \kappa, \beta < o^{\mathcal{U}}(\alpha)\rangle$ is a coherent sequence if for every $\alpha \leq \kappa, \beta < o^{\mathcal{U}}(\alpha)$, \[j_{U_{\alpha,\beta}}(\mathcal{U}) \restriction (\alpha, \beta) = \mathcal{U} \restriction (\alpha, \beta) \] 
where $\mathcal{U} \restriction (\alpha,\beta) = \langle U_{\gamma,\delta} \mid (\gamma < \alpha, \beta < o^{\mathcal{U}(\gamma)}) \vee (\gamma = \alpha \wedge \delta < \beta)\rangle$.
\end{definition}

\begin{definition}
Let $\vec{U}$ be a Mitchell increasing sequence of normal measures on $\kappa$. We say that $\vec U = \langle U_\alpha \mid \alpha < \delta\rangle$ is pre-coherent if there is a function $c$ such that for every $\alpha < \delta$, $j_{U_\alpha}(c)(\kappa) = \alpha$.
\end{definition}
Any coherent sequence of normal measures gives rise to a pre-coherent Mitchell increasing sequence of normal measures of the same length, but not necessarily vice-verse, see \cite{Ben-Neria}.

\begin{remark}
Let $\vec U$ be a Mitchell increasing sequence of normal measures. If $\len \vec U < \kappa^+$ then $\vec U$ is pre-coherent.
\end{remark}
\begin{proof}
If $\len \vec U = 0$, there is nothing to prove. So let us assume that it is non-zero.

Since the measures are discrete, there is a sequence of pairwise disjoint sets $\langle A_\alpha \mid \alpha < \len \vec U\rangle$ such that $A_\alpha \in U_\beta \iff \alpha = \beta$. 

Define $g(\zeta) = \alpha$ iff $\zeta \in A_\alpha$ and $0$ otherwise. Let also fix a sujective function $r \colon \kappa \to \len \vec U$. Now, $a = \range j(r)\restriction \kappa = j\image \len \vec U$. In particular, for $\pi_a$ the Mostowski collapse of the set $a$, and it is easy to verify that $\pi_a \circ j(g) \in M$ is the desired function. 

Finally, let 
\[c(\zeta) = \pi_{\range r \restriction \zeta} (g(\zeta)),\]
where this application is defined, and zero otherwise.
\end{proof}
Let $\vec U$ be a sequence of normal measures on $\kappa$, increasing in the Mitchell order, and let $o^{\vec U}(\kappa) = \zeta$. Let us define an iteration as well as a sequence of ordinals $\langle \gamma_\xi\mid \xi < \alpha_*\rangle$ as follows: 

$M_0 = V$, $j_{0,0} = id_V$. Given $M_{\alpha}$ and maps $j_{\beta,\alpha} \colon M_\beta\to M_\alpha$ for all $\beta < \alpha$, we pick the least index $\gamma_{\alpha} < \len j_{\alpha}(\vec{U})$ such that $\sup \{\beta < \alpha \mid j_{\beta, \alpha}(\gamma_\beta) = \gamma_\alpha\}$ is bounded (and we let $\gamma_0=0$). 

Let $M_{\alpha + 1} = \Ult(M_\alpha, j_{\alpha}(\vec U)(\gamma_\alpha))$ and let $j_{\alpha, \alpha+1}$ be the ultrapower map. Let $j_{\beta, \alpha + 1} = j_{\alpha,\alpha+1}\circ j_{\beta, \alpha}$ for all $\beta \leq \alpha$. 

If $\gamma_\alpha$ is undefined, we halt.

Let $\kappa_\alpha$ be the critical point of $j_{\alpha,\alpha+1}$ and let $\alpha_*$ be the length of the process. 

The following Lemma is a comparison argument due to Mitchell, in disguise.
\begin{lemma}
The iteration halts. Moreover, if $\len \vec{U} = \mu < \kappa$, it halts after $\omega^\mu$ many steps (ordinal exponentiation).
\end{lemma}
\begin{proof}
Let us assume towards a contradiction that the iteration continues for $\lambda = (\len(\vec U)^\kappa)^{+}$ many steps (cardinal exponentiation). 

Consider $\gamma_\alpha$ for limit ordinals $\alpha < \lambda$. As we take direct limits at limit steps of the iteration, each such $\gamma_\alpha$ is the image of an ordinal from a previous step in the iteration. More precisely, there are finitely many ordinals $\beta_0 < \dots < \beta_{n-1} < \alpha$ and a function $f \colon \kappa^n \to \len \vec U$ in $V$, such that $j_{\alpha}(f)(\kappa_{\beta_0}, \dots, \kappa_{\beta_{n-1}}) = \gamma_{\alpha}$. 

By Fodor's lemma, there is a stationary set $S \subseteq \lambda$, such that $f$ and the finite sequence $\beta_0, \dots, \beta_{n-1}$ are fixed. Pick $\alpha \in S \cap \acc S$. Then, for every $\beta \in S \cap \alpha$, $\gamma_{\alpha} = j_{\beta,\alpha}(\gamma_\beta)$, a contradiction to our choice of $\gamma_\alpha$.  

For the moreover part, one can verify by induction that for every non-zero ordinal $\alpha < \kappa$ of Cantor's normal form $\alpha = \omega^{\beta_0}\cdot n_0 + \omega^{\beta_1}\cdot n_1 + \cdots + \omega^{\beta_{m}} \cdot n_m$, where $\beta_0 > \beta_1 > \cdots > \beta_m$, and $n_i \neq 0$ for all $i$, $\gamma_\alpha = \beta_m$. Since the length of $\vec{U}$ in this case is below the critical point of $j_\alpha$, the process terminates at step $\alpha_* = \omega^\mu$. 
\end{proof}

Let $P = \langle \kappa_\alpha \mid \alpha < \alpha_*\rangle$, be the sequence of the critical points. In this case, the model $M_{\alpha_*}[P]$ clearly contains sets which are not in $M_\beta$ for large $\beta$. The reason is that for every infinite $\beta$, $M_\beta$ is not closed under countable sequences and in particular will not contain the initial segments of $P$.  

Thus, the correct models for this theorem are $M_{\alpha}[P \restriction \alpha]$. Note that for $\alpha$ limit, $\kappa_{\alpha}$ is (typically) a singular cardinal in this model. Thus, in those cases there is no elementary embedding from $M_{\alpha}[P\restriction \alpha]$ to $M_{\beta}[P \restriction \beta]$. 

\begin{lemma}
$M_{\alpha_*}[P] \subseteq \bigcap_{\alpha < \alpha_*} M_\alpha[P \restriction \alpha]$.
\end{lemma}
\begin{proof} 
Working in $M_{\alpha}[P\restriction \alpha]$ one can compute from the parameter $\gamma_\alpha$ the rest of the models $M_{\alpha'}$ for $\alpha' > \alpha$ and the iteration, and thus $P \setminus \alpha$. \end{proof}

\begin{theorem}
Let us assume that $\len \vec{U} \leq \kappa^+$ and that it is pre-coherent then $M_{\alpha_*}[P]$ is closed under $\kappa$-sequences.
\end{theorem}
\begin{proof}
Fix a function $c$ witnessing the pre-coherency of $\vec U$.

Let $\alpha \in \Ord$. Then, $\alpha$ is represented by some function $f \colon \kappa^n \to \Ord$ and a finite sequence $\zeta_0 < \cdots < \zeta_{n-1}$ in $P$, such that:
\[\alpha = j_{\alpha_*}(f)(\zeta_0,\dots, \zeta_{n-1}).\]
The main challenge is to "describe" which elements from $P$ we evaluate $j_{\alpha_*}(f)$ in, in order to obtain $\alpha$, in particular where $\len \vec U \geq \kappa$.



\begin{lemma}
Let $\zeta \in P$, $\beta = \otp(P\cap \zeta)$. 

Then, there is a function $g \in M_\beta$ and $\zeta' < \zeta$ in $P$ such that $\zeta = \min (\{\rho \in P \mid \rho > \zeta' \text{ and } j_{\beta,\alpha_*}(g)(\rho) = j_{\alpha_*}(c)(\rho)\}$.

In particular, there are $\zeta_0, \dots, \zeta_{n-1} < \zeta$ in $P$ and $h \colon \kappa^n \to o(\kappa)$ in $V$ such that 
\[\zeta = \min (\{\rho \in P \mid \rho > \zeta_{n-1} \text{ and } j_{\alpha_*}(h)(\zeta_0,\dots, \zeta_{n-1}, \rho) = j_{\alpha_*}(c)(\rho)\}.\]
\end{lemma}
\begin{proof}
For $\zeta \in P$ which is not a limit point, the claim is obvious: its Mitchell order under $\vec U$ is simply $0$ and one can read it from its predecessor.

For $\zeta$ limit, as $\gamma_\beta < \zeta^+$, there is a canonical function representing it, $g\in M_\beta$. As $\beta$ is limit, $g = j_{\bar\beta, \beta}(\bar g)$ for some $\bar g \in M_{\bar \beta}$, which is the canonical function for some ordinal $\bar\delta$. Being canonical, the ordinal $\bar\delta$ is definable from $\bar{g} = g \restriction \zeta'$ for some $\zeta'$ and in general, for every $\zeta' \leq \xi < \zeta$ in $P$, $g(\zeta)$ is the unique ordinal which is the height of the canonical function $g \restriction \xi$. Since $j_{\delta,\beta}(g \restriction \xi) = g$ for $\delta = \otp(P \cap \xi)$, $\xi \in P$, we conclude that if equality holds, then $j_{\delta,\beta}(\gamma_\xi) = \gamma_\beta$, but this can only hold for boundedly many values in $P$. 

The second part follows from the assumption that we are taking direct limits at limit steps and thus $g$ is the image of a function from a shorted iteration. Thus, by induction, one can represent $g$ as $j_{\beta}(h)(\zeta_0,\dots, \zeta_{n-1})$ (recall that $\beta = \otp(P\cap \zeta)$), $\zeta_0,\dots, \zeta_{n-1} < \zeta$ in $P$ and $h \colon \kappa^n \to o(\kappa)$. 
\end{proof}

Let  
\[\begin{matrix} j_{\alpha_*}(h)^P(\zeta_0,\dots,\zeta_{n-1}) & = & \min (\{\rho \in P \mid & \rho > \zeta_{n-1}  \text{ and }  \\ 
& & &j_{\alpha_*}(h)(\zeta_0,\dots, \zeta_{n-1}, \rho) = o^{j_{\alpha_*}(\mathcal{U})}(\rho)  \}. \end{matrix}\]

Given $\zeta \in P$ we may find an increasing sequence of elements of $P$, $\zeta_0, \dots, \zeta_n$ and functions $g_0, \dots, g_{n-1} \in V$ such that: 
\[\zeta_k = j_{\alpha_*}(g_k)^P(\zeta_0,\dots, \zeta_{k-1})\]

This is always possible, using standard arguments, see \cite{GitikKaplan}.

So, we conclude that each ordinal $\alpha$ can be represented using a function $f$ and finitely many functions $g_0,\dots, g_{n-1}$ that allows us to "read" the relevant elements from $P$. 

Let $\langle \alpha_i \mid i < \kappa\rangle$ be a sequence of ordinals. Let $\langle f^i, g_0^i,\dots, g_{n_i - 1}^i \mid i < \kappa\rangle$ be a choice of representatives, as above. Namely, take $\beta^i_j = j_{\alpha_*}^P(\beta^i_0,\dots, \beta^i_{j-1})$ and $\alpha_i = j_{\alpha_*}(f)(\beta^i_0,\dots, \beta^i_{n_i-1})$. 

Apply $j_{\alpha_*}$, and truncate at $\kappa$, and then compute from $P$ first the indexes $\beta^0_i, \dots, \beta^{n-1}_i$ and then the ordinals $\alpha_i$.  
\end{proof}

Let us analyse $\cf \alpha_*$ (which is the cofinality of $j_{\alpha_*}(\kappa)$), under the assumption that $\len \vec U < \kappa^+$.
\begin{lemma}
Let us assume that $\vec U$ is weakly-coherent and that $\cf \len \vec U \leq \kappa$.
 
Let $\lambda = \cf^{M_{\alpha_*}[P]}(j_{\alpha_*}(\kappa))$. Then, 
\[\lambda = \begin{cases} \omega & \len \vec U \text{ is a successor ordinal} \\
\cf^{V}( \len \vec U) & \cf^{V}( \len \vec U) < \kappa \\
\omega &  \cf^{V}( \len \vec U) = \kappa 
\end{cases}\]
\end{lemma}
\begin{proof}
We split into cases.

{\bf Case 0:} $\len \vec U$ is a successor ordinal. In this case $\cf \alpha_*$ is $\omega$. Indeed, the ordinals $\alpha$ in which $\gamma_{\alpha}$ is the last measure in $j_\alpha(\vec U)$ form an $\omega$-sequence, $\alpha_n$ cofinal at $\alpha_*$. Using the pre-coherency of the sequence, $\gamma_\alpha$ can be read from $P$.

{\bf Case 1:} $\cf \len \vec U = \mu < \kappa$. Pick a cofinal sequence $\langle \rho_i \mid i < \mu\rangle$, cofinal at $\len U$. Since $\mu < \kappa$ which is the critical point, for all $\alpha$ (and in particular, for $\alpha_*$), $\cf^{M_{\alpha}} \len j_{\alpha}(\vec U) = \mu$. 

Consider $\alpha_*$. For each $i$, there are unboundedly many $\beta < \alpha_*$ such that $\gamma_\beta = j_{\beta}(\rho_i)$. Let $\beta_i$ be the least such ordinal. Then, $\beta_i$ are increasing, and cofinal at $\alpha_*$. Indeed, below each $\beta_i$, the images of $\rho_j$, for all $j < i$ as well as all lower ordinals appear unboundedly. Again, using the pre-coherency, this cofinal sequence at $j_{\alpha_*}(\kappa)$ can be computed in $M_{\alpha_*}[P]$. The cofinality cannot be lower than that, as otherwise, the cofinality of $\mu$ must be collapsed, but $M_{\alpha_*}[P] \subseteq V$.

{\bf Case 2:} $\cf^V \len \vec U = \kappa$. In this case we need to show that $\cf \alpha_* = \omega$. 

Fix a sequence $\vec \delta = \langle \delta_\xi \mid \xi < \kappa\rangle$ cofinal at $\len\vec U$. Let us define by induction a sequence of ordinals cofinal at $\alpha_*$. 

Let $\alpha_0 = 0$. 

Let $n < \omega$. Denote $\rho_n = j_{\alpha_n + 1}(\vec \delta)(j_{\alpha_n}(\kappa))$. Since $\rho_n < \len\vec U$, we may define $\alpha_{n+1}$ to be the least ordinal such that $\gamma_{\alpha_{n+1}} = j_{\alpha_n+1,\alpha_{n+1}}(\rho_n)$. 

Let us show that $\alpha_* = \sup \alpha_n$. Indeed, let $\alpha_\omega = \sup \alpha_n$, and let us assume towards a contradiction, that $\alpha_\omega \neq \alpha_*$, so $\gamma_{\alpha_*}$ is an ordinal $<j_{\alpha_\omega}(\len\vec U)$. 

Since $\alpha_\omega$ is a limit ordinal, $\kappa_{\alpha_{\omega}} = \sup_{n<\omega} \kappa_{\alpha_n}$ and thus $j_{\alpha_\omega}(\vec \delta)$ is cofinal at $j_{\alpha_\omega}(\len\vec U)$, and there is $n$ such that the $\kappa_{\alpha_n}$-th point of this sequence exceeds $\gamma_{\alpha_\omega}$. Without loss of generality, $\gamma_{\alpha_\omega} = j_{\alpha_n,\alpha_\omega}(\bar\gamma)$. 

Now, as for every $m > n$, $\rho_{m} > j_{\alpha_n,\alpha_{m}}(\bar\gamma)$, there must be an ordinal $\beta$ between $\alpha_m$ and $\alpha_{m+1}$ such that $j_{\beta,\alpha_{\omega}}(\gamma_\beta) = \gamma$, contradicting the definition of $\gamma_{\alpha_\omega}$.
\end{proof}
We conclude that if $\len\vec U < \kappa^+$ then $\cf \alpha_* < \kappa$. Moreover, this can be decoded from $P$ itself, and thus this cofinality is going to be correctly computed in $M_{\alpha_*}[P]$. 

\begin{lemma}
If $\cf^V(\len \vec U) = \kappa^+$ then $M_{\alpha_*}[P]\models j_{\alpha_*}(\kappa)$ is regular.
\end{lemma}
\begin{proof}
We will prove the claim for the case $\len \vec U = \kappa^+$. The general case is similar.
\begin{claim}\label{claim:chain-condition-radin}
for every $\alpha$, $M_{\alpha}[P\restriction \alpha]$ is a $\kappa_{\alpha}^+$-c.c.\ extension of $M_\alpha$ and in particular $\kappa_\alpha^+$ is preserved.
\end{claim}
The proof for this claim is very similar to the proof of Lemma \ref{lemma:chain-condition-Prikry}, with the additional complication that every function in $M_{\beta}[P\restriction \beta]$ needs to be bounded (using an inductive hypothesis) by a corresponding function from $M_{\beta}$. \footnote{There is a different way to show that the regularity of $\kappa_\alpha^+$ is preserved, by showing that $\kappa_\alpha$ must be singular in $M_\alpha[P \restriction \alpha]$ (using an inductive hypothesis and the previous theorem) and thus the cofinality of this cardinal must be below $\kappa_\beta$ for some $\beta < \alpha$. Similar argument appears ahead.}

Let us assume now that $\cf^{M_{\alpha_*}[P]}(j_{\alpha_*}(\kappa)) < j_{\alpha_*}(\kappa)$. 
So, there is $\alpha < \alpha_*$ such that $\cf^{M_{\alpha_*}[P]}(j_{\alpha_*}(\kappa)) < \kappa_{\alpha}$ and thus $M_{\alpha}[P\restriction \alpha] \models \cf \alpha_* < \kappa_\alpha$.

In this model, $M_{\alpha}[P\restriction \alpha]$, 
one can compute the rest of the iteration $j_{\alpha, \alpha_*}$ and in particular, 
the sequence $\langle \gamma_{\alpha} \mid \alpha < \alpha_*\rangle$. It is clear that this sequence is cofinal at $j_{\alpha_*}(\kappa^+) = \len j_{\alpha_*}(\vec U)$. 

As $P$ is cofinal at $j_{\alpha_*}(\kappa)$, $\cf^{M_{\alpha}[P\restriction \alpha]}(\alpha_*) = \cf^{M_{\alpha}[P\restriction \alpha]}(j_{\alpha_*}(\kappa)) < \kappa_{\alpha}$.

Since the embedding $j_{\alpha, \alpha_*}$ is continuous at $j_{\alpha_*}(\kappa^+)$, we conclude that the cofinality of $j_{\alpha_*}(\kappa^+)$ in $M_{\alpha}[P\restriction \alpha]$ must be the same as the cofinality of $j_{\alpha}(\kappa^+)$ which is strictly larger than $\kappa_{\alpha}$, by the chain condition of the forcing. 

But this is a contradiction --- on the one hand the cofinality of $\alpha_*$ must be strictly below $\kappa_\alpha$ and on the other hand it must be $j_{\alpha}(\kappa^+)$.\end{proof}


\begin{lemma}\label{lemma:lifting-embedding}
Let $\vec U$ be a Mitchell increasing sequence of measures and let us consider the corresponding iteration. 

If $\alpha < \alpha_*$ satisfies that $\gamma_\alpha$ is strictly larger than $j_{\beta,\alpha}(\gamma_\beta)$ for all $\beta < \alpha$, then the embedding $j_{\alpha + 1, \alpha_*} \colon M_{\alpha + 1} \to M_{\alpha_*}$ lifts to an embedding $\tilde{j}_{\alpha+1,\alpha_*} \colon M_{\alpha + 1}[P\restriction \alpha] \to M_{\alpha_*}[P\restriction \alpha]$. 
\end{lemma}
\begin{proof}
Without loss of generality, $\gamma_\alpha \in j_{\alpha}''\len\vec U$. Otherwise, we will need to repeat the following process finitely many times. 

Let $\bar\gamma$ be an ordinal such that $j_{\alpha}(\bar\gamma) = \gamma_\alpha$. 

Let us consider the ultrapower by $U_{\bar{\gamma}}$, $k_0 \colon V \to N$. By the definition, in this model, the sequence $\vec U \restriction \bar \gamma$ exists. By our choice of $\gamma_\alpha$, and as $P(\kappa), \vec U \restriction \bar\gamma \in N$, if we will start to iterate $N$ based on $\vec U \restriction \bar\gamma$ we will obtain exactly the iteration $j_{\alpha} \restriction N$ and $j_{\alpha}(\bar\gamma) = \gamma_\alpha$. 

Next, the following diagram commutes:
\begin{center}
\begin{tikzcd}
V\ar[r,"k_0"]\ar[d, "j_\alpha"] & N\ar[d, "j_\alpha"]\ar[r,"k_{\alpha+1,\alpha_*}"] & 
N_{\alpha_*}\ar[d, "j_{\alpha}"]\\ 
M_{\alpha}\ar[r, "j_{\alpha,\alpha+1}"] & M_{\alpha+1}\ar[r, "j_{\alpha+1,\alpha_*}"] & M_{\alpha_*}
\end{tikzcd} 
\end{center}
where $k_{\alpha+1, \alpha_*}$ is the iteration as defined in $N$ using $k_0(\vec U)$. 

In $N$, one can compute $P\restriction \alpha$ and as $k_0(\kappa)$ is an inaccessible much larger than $\kappa$ in $N$, this set is bounded below $k_0(\kappa)$. Thus, the embedding $k_{\alpha+1,\alpha_*}$ can be restricted to the class $M_{\alpha+1}[P\restriction \alpha]$. 
Moreover, since $k_{\alpha+1, \alpha_*}\restriction M_{\alpha+1} = j_{\alpha+1,\alpha_*}$, we conclude that the image of this map is going to be contained in $M_{\alpha_*}[P\restriction \alpha]$. 

In order to show elementarity, we recall the definition \[M_{\alpha + 1}[P\restriction \alpha] = \bigcup_{\zeta\in \Ord} L[M_{\alpha+1}\cap V_\zeta, P\restriction \alpha]\]
\[M_{\alpha_*}[P\restriction \alpha] = \bigcup_{\zeta\in \Ord} L[M_{\alpha_*}\cap V_\zeta, P\restriction \alpha]\]  
so the restriction of $k_{\alpha+1,\alpha_*}$ to each component is elementary, and thus it is elementary.
\end{proof}

\begin{theorem}\label{thm:BD-for-short-radin}
Let $\vec U$ be a Mitchell increasing sequence of measures and let us assume that $\len \vec U <\kappa^+$. Then $\bigcap M_\alpha[P\restriction \alpha] = M_{\alpha_*}[P]$.
\end{theorem}
\begin{proof}
Let us assume by induction on $\kappa$ and $\len \vec U$ that the theorem holds, namely that $\bigcap M_{\alpha}[P \restriction \alpha] = M_{\alpha_*}[P]$. Let $\beta_i$ the cofinal sequences at $\alpha_*$ defined in the cases above.

{\bf Case 0:} $\len \vec U$ is a successor ordinal. In this case, apply the inductive hypothesis for $\len \vec U$, (and elementarity) in the model $M_{\beta_n + 1}[P \restriction \beta_n]$. 
We get that $M_{\beta_{n+1}}[P \restriction \beta_{n+1}] = \bigcap_{\alpha \in [\beta_n, \beta_{n+1})} M_{\alpha}[P\restriction \beta_n][P \restriction [\beta_n, \alpha)]$. 
This is true, by Lemma \ref{lemma:lifting-embedding}, applied finitely many times. So, in order to show that the theorem holds, consider $X \in \bigcap_{n<\omega} M_{\beta_n + 1}[P \restriction \beta_n]$, and use the elementary emebddings $\tilde j_{\beta_n + 1, \alpha_*} \colon M_{\beta_n+1}[P \restriction \beta_n] \to M_{\alpha_*}[P\restriction \beta_n]$, and repeat the argument of Theorem \ref{thm:BD-for-vanilla}. 

{\bf Case 1:} $\cf \len \vec U < \kappa$. This case is the same, we notational differences, using the closure of the models $M_\alpha[P]$ under $\kappa$-sequences.

{\bf Case 2:} $\cf \len \vec U = \kappa$. In this case, $\cf \alpha_* = \omega$ and we can repeat the argument of Case 0.
\end{proof}
It worth mentioning that the intersection theorem is quite weak in this case. Indeed, if we would apply it for an arbitrary sequence of measures, which might fail to be Mitchell increasing, it still holds, but it might be quite degenerated. For example, if we look at an iteration of length $\omega + 1$ of the same measure $U$ and look at the model $M_{\omega + 1}[P]$, then since $P \restriction \omega$ defines the normal measure $j_{\omega}(U)$, this model is going to be simply $M_\omega[P]$ (in particular, class many cardinals of $M_{\omega+1}$ are collapsed in $M_{\omega+1}[P]$). 

\begin{claim}
Let $\vec U$ be a Mitchell increasing sequence of measures of length $\kappa^+$. 

Then $\bigcap_{\alpha < \alpha_*} M_{\alpha}[P \restriction \alpha]$ is strictly larger than $M_{\alpha_*}[P\restriction \alpha_*]$.
\end{claim}
\begin{proof}
Let us show that $\Gamma = \langle j_{\alpha,\alpha_*}(\gamma_\alpha) \mid \alpha < \alpha_*\rangle$ belongs to the intersection model, but not to the generic extension. 

First, let us show by induction on $\beta$ for every sequence of measures on $\kappa$ of length $\beta < \kappa^+$, $\vec{U}'$ the corresponding sequence $\Gamma'$ belongs to $M'_{\beta_\star}[P']$ (where all those objects are defined using $\vec{U}'$. 

Let us assume that the claim is proved for every $\beta' < \beta$, and let us consider the case of a sequence of length $\beta$. Since $M_{\beta_*}[P] = \bigcap_{\delta < \beta_*} M_{\delta }[P\restriction \delta]$, it is enough to show that for every $\delta < \beta_*$, $\Gamma \in M_{\delta}[P \restriction \delta]$.

We prove that $\Gamma \in M_{\delta}[P\restriction \delta]$ using a second level of induction, on $\delta < \beta_*$. Note that since $\Gamma \restriction [\delta, \beta_*) \in M_\delta$, it is enough to show that $\Gamma \restriction \delta \in M_\delta[P\restriction \delta]$.

Let $\rho < \delta$ be the last ordinal such that $j_{\rho,\delta}(\gamma_\rho) \geq \gamma_{\delta}$, assuming that there is an ordinal $\rho$ such that $j_{\rho,\delta}(\gamma_\rho) \geq \gamma_{\delta}$. If there is such an ordinal then there is a maximal one, since below every $\rho$ which is large enough so that $\gamma_\delta$ will be in the image of $j_{\rho,\delta}$, and $j_{\rho,\delta}(\gamma_\rho) > \gamma_\delta$ there are unboundedly many ordinals such that $j_{\rho,\delta}(\gamma_{\rho'}) = \gamma_{\delta}$.  Using the definition of $\gamma_\delta$, we know that this set is bounded and using the definition of $\gamma_\rho$, we know that it is closed.

By the inductive hypothesis, $\Gamma \restriction \delta \in M_{\rho}[P \restriction \rho]$. In $M_{\rho + 1}$, the iteration up to $\delta$ is definable using the measure sequence $j_{0, \rho + 1}(\vec{U} \restriction \gamma_\rho)$, which by the (external) induction hypothesis, pushed forward by $j_{0,\rho+1}$, satisfies that $M_{\delta}[P\restriction [\rho, \delta)$ contains $\Gamma\restriction [\rho, \delta)$. Combining all together, the result follows.

Next, let us verify that $\Gamma \notin M_{\alpha_*}[P]$. The map $\kappa_\alpha \mapsto \gamma_\alpha$ is a surjection on $j_{\alpha_*}(\len \vec U)$ which is $j_{\alpha_*}(\kappa^+)$, and in particular $j_{\alpha_*}(\kappa^+)$ is collapsed. But, this is impossible, by Claim \ref{claim:chain-condition-radin}. 
\end{proof}

\begin{question}
Is $M_{\alpha_*}[P]$ is always closed under $\kappa$-sequences?
\end{question}

\begin{question}
Is there a parallel for the intersection theorem for Radin forcing with $o(\kappa) \geq \kappa^+$? 
\end{question}

\providecommand{\bysame}{\leavevmode\hbox to3em{\hrulefill}\thinspace}
\providecommand{\MR}{\relax\ifhmode\unskip\space\fi MR }
\providecommand{\MRhref}[2]{%
  \href{http://www.ams.org/mathscinet-getitem?mr=#1}{#2}
}
\providecommand{\href}[2]{#2}

\end{document}